\RequirePackage{etex}
\RequirePackage{easymat}

\documentclass[12pt]{elsarticle}
\usepackage{amsmath, amsthm,dsfont,mathrsfs, amssymb,mathabx}
\usepackage[active]{srcltx}

\usepackage[all]{xy}

\usepackage[T1]{fontenc}

\newtheorem{theorem}{Theorem}
\newtheorem{lemma}{Lemma}

\theoremstyle{remark}

\newtheorem{example}{Example}

\theoremstyle{definition}
\newtheorem{definition}{Definition}

\newcommand{\mc}{\mathcal}
\newcommand{\kv}{$^{\scriptscriptstyle\boxplus}$}
\newcommand{\cc}{\mathbb C}
\newcommand{\rr}{\mathbb R}

\newcommand{\un}{\underline}

\renewcommand{\l}{\langle}
\renewcommand{\r}{\rangle}

\renewcommand{\le}{\leqslant}
\renewcommand{\ge}{\geqslant}
\newcommand{\matt}[1]{\left[\begin{smallmatrix}
 #1 \end{smallmatrix}\right]}
\newcommand{\mat}[1]{\begin{bmatrix}
 #1  \end{bmatrix}}
\newcommand{\ma}[1]{\begin{matrix}
 #1  \end{matrix}}
\newcommand{\arr}[1]{\left[\begin{array}
 #1\end{array}\right]}

\newcommand{\arrr}[4]{\left[\begin{array}{c|c}
 #1&#2\\\hline #3&#4\end{array}\right]}

\begin{document}
\title{Operators on positive semidefinite inner product spaces\footnote{This is a preprint of the paper published in Linear Algebra Appl. 596 (2020) 82--105.}}

\author[bov]{Victor A.~Bovdi} \ead{vbovdi@gmail.com}
\address[bov]{United Arab Emirates University, Al Ain, UAE}

\author[kly]{Tetiana Klymchuk}
\ead{klimchuk.tanya@gmail.com}
\address[kly]{Universitat Polit\`{e}cnica de Catalunya, Barcelona, Spain}

\author[ser]{Tetiana Rybalkina} \ead{rybalkina\_t@ukr.net}

\author[bov]{Mohamed A.~Salim}
\ead{msalim@uaeu.ac.ae}

\author[ser]{Vladimir~V.~Sergeichuk\corref{cor}}
\ead{sergeich@imath.kiev.ua}
\address[ser]{Institute of Mathematics, Tereshchenkivska 3,
Kiev, Ukraine}

\cortext[cor]{Corresponding author.}

\begin{abstract}
Let $U$ be a semiunitary space; i.e., a complex vector space with scalar product given by a positive semidefinite Hermitian form $\l\cdot,\cdot\r$. If a linear operator  ${\cal A}: U\to U$ is bounded (i.e., $\|{\cal A}u\|\le c\|u\|$ for some $c\in\rr$ and all $u\in
U$), then the subspace $U_0:=\{u\in U\,|\,\l u,u\r=0\}$ is invariant, and so $\mc A$ defines the linear operators
$\mc A_0: U_0\to U_0$ and $\mc A_1: U/U_0\to U/U_0$.

Let $\cal A$ be an indecomposable bounded operator on $U$ such that $0\ne U_0\ne U$. Let $\lambda$ be an eigenvalue of ${\cal A}_0$.  We prove that the algebraic multiplicity of $\lambda $ in ${\cal A}_1$ is not less than the geometric multiplicity of $\lambda $ in ${\cal A}_0$, and   the geometric multiplicity
of $\lambda $ in ${\cal A}_1$  is not less than  the number of Jordan blocks
$J_t(\lambda ) $ of each fixed size $t\times t$ in the Jordan canonical form of ${\cal
A}_0$.

We give canonical forms of selfadjoint and isometric operators on $U$, and of Hermitian forms on $U$.

For an arbitrary system of semiunitary spaces and linear mappings on/between them,
we give an algorithm that reduces their matrices to canonical form. Its special cases are Belitskii's and Littlewood's algorithms for systems of linear operators on vector spaces and unitary spaces, respectively.
\end{abstract}


\maketitle

\section{Introduction}

We study linear operators on  a complex vector space $U$ with scalar product given by a positive semidefinite Hermitian form $\l\cdot,\cdot\r:U\times U\to \cc$. We suppose that $\l\cdot,\cdot\r$ is semilinear in the first argument and linear in the second.

There exists a basis of $U$, in which the scalar product is given by the matrix
\begin{equation}\label{gru}
H_m:=\mat{I_m&0\\0&0};
\end{equation}
we call this basis \emph{$m$-orthonormal} or \emph{semiorthonormal}.
If $(\alpha _1,\dots, \alpha_n)$ and
$( \beta_1,\dots, \beta_n)$ are the coordinates of $u,v\in U$ with respect to an $m$-orthonormal basis $e_1,\dots,e_n$, then
$$ \l u,v\r=
\bar{\alpha}_1 \beta_1+\dots+ \bar{\alpha}_m \beta_m. $$

Bognar \cite[p.\,14]{bog} calls $U$ a  \emph{positive semidefinite inner product space}. We call $U$  a \emph{semiunitary space} for short and basing on the fact that  $U$ is a direct sum of a \emph{unitary space} (i.e., a complex vector space with positive definite Hermitian form)  and a complex vector space with zero scalar product.

In Section \ref{ddc} we formulate the main theorems. In Section \ref{vgw} we give some reduced form of a matrix of a bounded linear operator, which is used in the next sections.
In Section \ref{hhgy} we prove Theorem \ref{t2.2} about eigenvalues of bounded linear operators on semiunitary spaces. In Section \ref{v2v} we prove Theorem \ref{ttt} about canonical matrices of
seminormal operators (their special cases are selfadjoint operators and isometric operators) and Hermitian forms on semiunitary spaces. In Section \ref{z3v} we construct an algorithm for reducing to canonical form of matrices of a system of linear mappings on semiunitary spaces; this section can be read independently of the other sections.

Mehl and Rodman \cite{meh2} study selfadjoint linear operators on a complex vector space with singular Hermitian form; they give a canonical form of selfadjoint operators on semiunitary spaces, which we also give in Theorem \ref{ttt}(a$_1$).
Many researches study semidefinite subspaces that are invariant under a certain linear operator on an indefinite inner product space $U$. Their results and many applications are given in \cite{goh,meh1} if $U$ is nondegenerate, in \cite{meh} if $U$ is possibly degenerate, and in \cite{azi,bog} if $U$ is a Krein space.

All vector spaces that we consider are finite dimensional complex vector spaces and all matrices are complex matrices.

\section{Main results}\label{ddc}

\subsection{Eigenvalues of bounded operators}

The \emph{isotropic part} of a semiunitary space $U$ is the subspace
\begin{equation}\label{rbg}
U_0:=\{u\in U\,|\,\l u,u\r=0\}.
\end{equation}

The \emph{orthogonal direct sum} $V\obot W$  of semiunitary spaces $V$ and $W$ is the semiunitary space $V\oplus W$ with scalar product
\[
\l v+w,v'+w'\r:=\l v,v'\r+\l w,w'\r\qquad \text{for $v,v'\in V$ and $w,w'\in W$}. Аня
\]

The \emph{direct sum} of two linear operators $\mc B:V\to V$ and $\mc C:W\to W$ on semiunitary spaces is the linear operator
\begin{equation}\label{cer}
\mc B\oplus \mc C:V\obot W\to V\obot W, \qquad v+w\mapsto \mc Bv+ \mc C w.
\end{equation}
An operator $\mc A:U\to U$ is \emph{indecomposable} if $U$ cannot be decomposed into an orthogonal direct sum of two invariant subspaces of smaller dimensions.

It follows from \cite[Theorem 2]{ser_izv} that each operator $\mc A:U\to U$ is decomposed into a direct sum of indecomposable operators
\begin{equation}\label{ine}
\mc A_1\oplus\dots\oplus \mc A_t:U_1\obot\dots\obot U_t\to U_1\obot\dots\obot U_t;
\end{equation}
this sum is determined uniquely, up to permutation of summands and replacement of each summand
$\mc A_i:U_i\to U_i$ by $\mc B_i:V_i\to V_i$ such that there is a linear bijection $U_i\to V_i$ that preserves the scalar products and transforms $\mc A_i$ to $\mc B_i$.

The \emph{length} of $u\in U$ is the number
$
\|u\|:=\sqrt{\l u,u\r}.
$
A linear operator  ${\cal
A}: U\to U$ is {\it bounded} if  there exists a positive $c\in\rr$ such that
\begin{equation}\label{abe}
\|{\cal A}u\|\le c\|u\|\qquad\text{for all }u\in
U.
\end{equation}

\begin{lemma} \label{nlp}
The following statements are equivalent for every operator ${\cal A}$ on a semiunitary space $U$:
\begin{itemize}
  \item[{\rm(a)}] The operator $\cal A$ is bounded.

  \item[{\rm(b)}] The subspace $U_0$ defined in \eqref{rbg} is invariant.

  \item[{\rm(c)}] The matrix of $\mc A$ in each $m$-orthonormal basis has the lower block triangular form
\begin{equation}\label{sbn}
A=\mat{B&0\\C&D},\quad
\text{in which $B$ is }m\times m.
\end{equation}
\end{itemize}
\end{lemma}

\begin{proof}
(a)$\,\Rightarrow\,$(b) \ Let the operator $\cal A$ be bounded, and let $u\in U_0$. By \eqref{abe},  $\|\mc A u\|\le c\|u\|=c0=0$. Hence $\mc A u\in U_0$.

(b)$\,\Rightarrow\,$(c) \ This statement is obvious.

(c)$\,\Rightarrow\,$(a) \  Let \eqref{sbn} be the matrix of $\mc A$ in an $m$-orthonormal basis, let  $[\alpha _1\,\dots\,\alpha_n]^T$ be the coordinate vector of $u\in U$, and let $a:=[\alpha _1\,\dots\,\alpha_m]^T$.  Then
$\|u\|=|a|:=\sqrt{\bar \alpha _1\alpha _1+\dots+\bar \alpha _m\alpha _m}$  and $\|\mc Au\| =|Ba|$. By \cite[Example 5.6.6]{hor},
 $|Ba| \le c|a|$, in which $c\in\rr$ is
the spectral norm of $B$, and so $\|\mc Au\|\le c\|u\|$.
\end{proof}

If $\mc A: U\to U$ is a bounded linear operator, then $U_0$ is an invariant subspace, and so
$\mc A$ defines the linear operators
\begin{equation}\label{asz}
\mc A_0: U_0\to U_0,\qquad \mc A_1: U/U_0\to U/U_0,
\end{equation}
where $\mc A_0$ is the restriction of $\mc A$ on $U_0$ and $\mc A_1(u+U_0):=A_1u+U_0$.
If \eqref{sbn} is the matrix of $\mc A$ in an $m$-orthonormal basis, then $B$ and $D$ are the matrices of $\mc A_1$ and $\mc A_0$.

Recall that the \emph{algebraic multiplicity} of an eigenvalue $\lambda $ of a linear operator $\mc A$ is
the multiplicity of $\lambda $ as a root of the characteristic polynomial of $\mc A$. The \emph{geometric multiplicity} of  $\lambda $ is
the number of Jordan blocks with eigenvalue $\lambda $  in the Jordan canonical form of  $\mc A$.

The following theorem is proved in Section \ref{hhgy}.

\begin{theorem} \label{t2.2}
Let $\cal A$ be an indecomposable bounded
operator on a semiunitary space $U$. Let ${\cal A}_0$ and ${\cal A}_1$ be the operators defined in \eqref{asz}. Let $0\ne U_0\ne U$ (if $U_0=0$ or $U_0=U$, then $U$ is a unitary space or a vector space, respectively).
If $\lambda $ is an eigenvalue  of ${\cal A}_0$, then
\begin{itemize}
  \item[\rm(a)]
$\lambda $ is  an eigenvalue of ${\cal A}_1$,

  \item[\rm(b)]  the algebraic multiplicity of $\lambda $ in ${\cal A}_1$ is greater than or equal to the geometric multiplicity of $\lambda $ in ${\cal A}_0$, and

  \item[\rm(c)] the geometric multiplicity
of $\lambda $ in ${\cal A}_1$ is greater than or equal to the number of Jordan blocks
$J_k(\lambda ) $ of each fixed size $k\times k$ in the Jordan canonical form of ${\cal A}_0$.
\end{itemize}
\end{theorem}

\subsection{Canonical matrices of seminormal operators and Hermitian forms}

A \emph{\kv matrix} $A$ is a block matrix with two vertical strips and two horizontal strips, in which the diagonal blocks are square. We give its partition into blocks by writing $[A]_m$, in which $m\times m$ is the size of the first diagonal block.

We always consider the matrix of a linear operator ${\cal A}:U\to U$ in an $m$-orthonormal basis as a \kv matrix
 \begin{equation}             \label{gsi}
A=[A]_m=\arrr{A_{11}}{A_{12}}{A_{21}}{A_{22}},\quad \text{in which }A_{11} \text{ is }  m\times m.
\end{equation}

Let $\mc B:V\to V$ and $\mc C:W\to W$ be linear operators on semiunitary spaces. Let
\begin{equation}\label{okn}
B=\arrr{B_{11}}{B_{12}}{B_{21}}{B_{22}},\qquad
C=\arrr{C_{11}}{C_{12}}{C_{21}}{C_{22}}
\end{equation}
be their \kv matrices in an $m'$-orthonormal basis
$f_1,\dots,f_{n'}$ of $V$ and an $m''$-orthonormal basis
$g_1,\dots,g_{n''}$ of $W$. Then the \kv matrix of $\mc B\oplus  \mc C$ (see \eqref{cer}) in the $(m'+m'')$-orthonormal basis
\[
f_1,\dots,f_{m'},g_1\dots,g_{m''},f_{m'+1},
\dots,f_{n'},g_{m''+1},\dots,g_{n''}
\]
is the {\it block-direct sum} of $B$ and $C$:
\begin{equation}  \label{utj}
B\boxplus C :=
\left[\begin{array}{cc|cc}
 B_{11}&0&B_{12}&0\\
 0&C_{11}&0&C_{12}\\\hline
 B_{21}&0&B_{22}&0\\
 0&C_{21}&0&C_{22}\\
 \end{array}\right].
\end{equation}

Let $\mc A$ be a linear
operator on a semiunitary space $U$. A linear operator $\mc B:U\to U$ is \emph{adjoint} for $\mc A$ if
\begin{equation}\label{bks}
\l u,\mc Av\r=\l \mc B u, v\r\qquad \text{for all } u,v\in U.
\end{equation}

\begin{lemma}       \label{lk1}
Let $\mc A$ and $\mc B$ be linear
operators on a semiunitary space $U$. Let $\mc A_1$ and $\mc B_1$ be defined by \eqref{asz}.
\begin{itemize}
  \item[\rm(a)] The operator $\mc B$ is adjoint for $\mc A$ if
and only if $\mc A$ and $\mc B$ are bounded and
${\cal B}_1={\cal A}^*_1$.

\item[\rm(b)]
In matrix form, if $A=\matt{A_1&A_2\\A_3&A_4}$ and $B=\matt{B_1&B_2\\B_3&B_4}$ are the \kv matrices of $\mc A$ and $\mc B$ in a semiorthonormal basis, then the operator $\mc B$  is adjoint for $\mc A$ if and only if $A_2=B_2=0$ and $B_1=A_1^*$.
\end{itemize}
 \end{lemma}

\begin{proof} Since (a) and (b) are equivalent, it suffices to prove (b).
The equality \eqref{bks} holds if and only if $(I_m\oplus 0)A=B^*(I_m\oplus 0)$, if and only if $A_2=B_2=0$ and $B_1=A_1^*$.
\end{proof}

Let $\cal A$ be a linear operator  on a semiunitary space $U$. The operator $\cal A$ is {\it selfadjoint} if $\l{\cal A}u,v\r = \l u, {\cal
A}v\r$, and $\cal A$ is {\it metric} if $\l{\cal A}u, {\cal A}v\r=\l u,v\r$
for all $u,v\in U$.

\begin{lemma}\label{bvn}
A linear operator $\cal A$ on a semiunitary space $U$ is selfadjoint (respectively, metric) if and only if it is bounded and the operator ${\cal A}_1$ on the unitary space $U_1$ defined in \eqref{asz} is selfadjoint (respectively, metric).
\end{lemma}

\begin{proof}
Let $U$ be a semiunitary space.

$\Longrightarrow$.
Let  ${\cal A}:U\to U$ be an unbounded operator, and let $u\in U_0$ be such that ${\cal
A}u \notin U_0.$ Then ${\cal A}$ is not selfadjoint since $\l \mc Au,\mc Au\r\ne 0$ and $\l u,\mc A^2u \r= 0$; ${\cal A}$ is not metric since $\l u, u \r= 0$ and $\l \mc Au,\mc Au\r\ne 0$.

Hence, if ${\cal A}:U\to U$ is a selfadjoint or metric operator, then it is bounded and $U_0$ is its invariant subspace. If ${\cal A}$ is selfadjoint, then $\l{\cal A}(u+U_0),v+U_0\r =
\l\mc Au+\mc AU_0,v+U_0\r=
\l\mc Au,v\r=\l u,\mc A v\r=
 \l u+U_0, {\cal
A}(v+U_0)\r$
for all $u,v\in U$, and so ${\cal A}_1:U_1\to U_1$ is selfadjoint.
If ${\cal A}$ is metric, then $\l{\cal A}(u+U_0), {\cal A}(v+U_0)\r=\l u+U_0,v+U_0\r$
for all $u,v\in U$, and so ${\cal A}_1:U_1\to U_1$ is metric.

$\Longleftarrow$. Let ${\cal A}:U\to U$ be a bounded operator. If ${\cal A}_1$ is selfadjoint,
then $\l\mc Au,v\r=\l{\cal A}(u+U_0),v+U_0\r =
\l u+U_0, {\cal
A}(v+U_0)\r=\l u,\mc A v\r$, and so $\mc A$ is selfadjoint. If ${\cal A}_1$ is metric,
then $\l\mc Au,\mc Av\r=\l{\cal A}(u+U_0),\mc A (v+U_0)\r =
\l u+U_0, v+U_0\r=\l u, v\r$, and so $\mc A$ is metric.
\end{proof}

The following definition is natural in view of Lemma \ref{bvn}: an
operator $\cal A$ on a semiunitary space $U$ is {\it
seminormal} if $\cal A$ is bounded and ${\cal A}_1$ is a normal
operator on the unitary space $U_1$.
Let  $\mc A$ and $\mc B$ be linear
operators on a semiunitary space $U$. If $\mc B$ is adjoint for $\mc A$ and $\mc A\mc B=\mc B\mc A$, then $\mc A$ is seminormal; the converse is not true.

Write
\begin{equation}\label{zxc}
J_n(\lambda ):=\mat{\lambda &&&0\\1&\lambda \\&\ddots&\ddots\\0&&1&\lambda}\qquad
(n\text{-by-}n,\ \lambda \in\cc).
\end{equation}

The following theorem is proved in Section \ref{v2v}.

\begin{theorem}       \label{ttt}
{\rm(a)}
Let  $\cal A$ be a seminormal operator on a semiunitary
space. Then there exists a semiorthonormal basis, in which the \kv matrix of $\mc A$ is a block-direct sum, uniquely determined up
to permutation of summands, of\/ \kv matrices of the type
\begin{equation}  \label{3.1}
[J_n(\lambda)]_{l} \quad\text{in which }\lambda \in
{\mathbb C} \text{ and }l\in\{0,1\}
\end{equation}
(see \eqref{gsi}).
In particular,
\begin{itemize}
\item[{\rm(a$_1$)}]  if  $\cal A$ is a selfadjoint operator on a semiunitary
space, then there exists a semiorthonormal basis, in which the \kv matrix of $\mc A$ is a block-direct sum, uniquely determined up
to permutation of summands, of\/ \kv matrices of the types
\begin{equation*}  \label{3.1x}
[J_n(\lambda)]_1\ (\lambda \in
{\mathbb R}),\qquad [J_n(\mu)]_0\ (\mu \in
{\mathbb C});
\end{equation*}

\item[{\rm(a$_2$)}] if  $\cal A$ is a metric operator on a semiunitary
space, then there exists a semiorthonormal basis, in which the \kv matrix of $\mc A$ is a block-direct sum, uniquely determined up
to permutation of summands, of\/ \kv matrices of the types
\begin{equation*}  \label{3a.1}
[J_n(\lambda)]_1\ (\lambda \in
{\mathbb C}\text{ and }|\lambda|=1),\qquad [J_n(\mu)]_0\ (\mu \in
{\mathbb C}).
\end{equation*}
\end{itemize}

{\rm(b)}
Let  $\cal F$ be a Hermitian form on a semiunitary
space. Then there exists a semiorthonormal basis,
in which the \kv matrix of $\cal F$ is a block-direct sum, uniquely determined up to
permutation of summands, of \kv matrices of the types
\begin{equation}  \label{3.2}
\arrr{0}{1}{1}{0}_1, \quad [\lambda]_1\ (\lambda\in \mathbb R),\quad  [1]_0,\quad [-1]_0,\quad [0]_0.
\end{equation}
\end{theorem}

The statement (a$_1$) is \cite[Theorem 13]{meh2}.

\subsection{An algorithm for reducing to canonical form of matrices of systems of linear mappings on semiunitary
spaces}
\label{z3va}

Let $A=[A_{ij}]_{i,j=1}^2$ and $B=[B_{ij}]_{i,j=1}^2$ be \kv matrices, in which $A_{11}$ and $B_{11}$ are $m\times m$. We say that $A$ and $B$ are \emph{semiunitarily similar} if $S^{-1}AS=B$ for some nonsingular matrix of the form
\begin{equation}           \label{fem}
S=\mat{ S_{11}&0\\
S_{21}&S_{22}},\quad\text{in which $S_{11}$ is an $m\times m$ unitary matrix.}
\end{equation}

In Section \ref{z3v}, which can be read independently of Sections \ref{vgw}--\ref{v2v}, we give an algorithm that reduces a matrix
of a bounded operator on a semiunitary space to canonical form. This algorithm reduces each \kv matrix $A$ to some \kv matrix $A_{\text{can}}$ by transformations of semiunitary similarity (see \eqref{jdf}) such that
\begin{quote}
two \kv matrices $A$ and $B$ are semiunitarily similar if and only if $A_{\text{can}}=B_{\text{can}}$.
\end{quote}
Moreover, this algorithm reduces to canonical form matrices of an arbitrary finite system of semiunitary spaces and linear mappings between them, which we consider as a semiunitary representation of a quiver (see Section \ref{n4d}). Thus, a \kv matrix $A$ is canonical if the algorithm acts identically on it. Unfortunately, it is impossible to describe explicitly the set of canonical forms since this description would contain, in particular, the canonical matrices of each system of linear mappings.
Even the problem of classifying linear operators on a unitary space is considered as hopeless because it contains the problem of classifying an arbitrary system consisting of several unitary spaces and linear mappings between them; i.e., the problem of classifying unitary representations of an arbitrary quiver (see \cite{ser1,ser2}).
Special cases of this algorithm are
Littlewood's and Belitskii's algorithms.

Littlewood's algorithm \cite{lit} reduces a square matrix to canonical form under unitary similarity. It was rediscovered in \cite{ben,ser1}; see the survey \cite{sha}. Unitary and Euclidean representations of quivers are studied in \cite{ser1,ser2} using Littlewood's algorithm.

Belitskii's algorithm \cite{bel,bel1,ser_can} reduces to canonical form the matrices of an arbitrary system of linear mappings. The geometric form of Tame-Wild Theorem \cite{gab_vos} was formulated and proved in \cite{ser_can} in terms of Belitskii's canonical matrices.
Belitskii's algorithm is used in \cite{br,ch,ch1,gal}.

\section{Matrices of bounded linear operators}\label{vgw}

\begin{lemma}\label{smc}
Let $e_1,\dots,e_n$ be an $m$-orthonormal basis of a semiunitary space $U$, and let $f_1,\dots,f_n$ be any basis of $U$. Then $f_1,\dots,f_n$ is $m$-orthonormal if and only if the change of coordinates matrix has the form
\eqref{fem}.
\end{lemma}

\begin{proof}
$\Longrightarrow$. Let the basis $f_1,\dots,f_n$ be $m$-orthonormal. Let the change of coordinates matrix be partitioned as follows: $S=[S_{ij}]_{i,j=1}^2$, in which $S_{11}$ is $m\times m$.
Then the block $S_{12}$ is zero since both $e_{m+1},\dots,e_n$ and $f_{m+1},\dots,f_n$ are bases of $U_0$. The block $S_{11}$ is unitary since $I_m\oplus 0_{n-m}$ is the matrix of the scalar product in both $e_1,\dots,e_n$ and $f_1,\dots,f_n$, and hence  $S^*(I_m\oplus 0_{n-m})S=I_m\oplus 0_{n-m}$, which implies that $S_{11}^*S_{11}=I_m$.

$\Longleftarrow$. If the change of coordinates matrix has the form \eqref{fem}, then $S^*(I_m\oplus 0_{n-m})S=I_m\oplus 0_{n-m}$, and so $f_1,\dots,f_n$ is an $m$-orthonormal basis.
\end{proof}

Let $B$ and $C$ be two  \kv matrices of the form \eqref{okn}, in which the blocks $B_{11}$ and $C_{11}$ are $m\times m$.
By Lemma \ref{smc},
\begin{equation}\label{jdf}
\parbox[c]{0.84\textwidth}{the problem of classifying linear operators on a semiunitary space is reduced to the problem of classifying \kv matrices up to semiunitary similarity.}
\end{equation}

We say that a \kv matrix is \emph{decomposable} if it is semiunitarily similar to a block-direct sum \eqref{utj} of \kv matrices of smaller sizes.  A linear operator is decomposable (see \eqref{cer}) if and only if its matrix is decomposable.

The second problem in \eqref{jdf} is reduced to the problem of classifying indecomposable \kv matrices up to semiunitary similarity since
\begin{equation}\label{frb}
\parbox[c]{0.84\textwidth}{each \kv matrix  is semiunitarily similar to a block-direct sum of indecomposable \kv matrices, and this sum is uniquely determined up to semiunitary similarity of summands}
\end{equation}
(compare with \eqref{ine}).
This statement is proved  by applying \cite[Theorem 2]{ser_izv} to a pair consisting of the Hermitian form $\l \cdot,\cdot\r$ and a linear operator because
by \cite[Theorem 2]{ser_izv} each system of sesquilinear forms and linear operators on a complex vector space is decomposed into a direct sum of indecomposable systems, and this sum is uniquely determined up to isomorphisms of summands.

For each $a=[\alpha _1\,\dots\,\alpha _k]^T\in \cc^{k\times 1}$, we define
\begin{equation*}\label{bfh}
|a|:=\sqrt{\bar \alpha _1\alpha _1+\dots+\bar \alpha _k\alpha _k}.
\end{equation*}
Therefore, if $(\alpha _1,\dots,\alpha _n)$ are the coordinates of a vector $u$ of a semiunitary space in an $m$-orthonormal basis and $a':=[\alpha _1\,\dots\,\alpha _m]^T$, then
\begin{equation}\label{qha}
\|u\|=|a'|.
\end{equation}

\begin{lemma}           \label{htp}
Let $\mc A$ be a bounded linear operator  on a semiunitary space. Let $\lambda _1,\dots,\lambda _s$ be all distinct eigenvalues of $\mc A_1$ in any prescribed order. Then the \kv matrix of $\mc A$ in some $m$-orthonormal basis has the form
\begin{equation}           \label{fri}
A=\arr{{ccc|ccc}B_{11}&&0\\&\ddots&&&0\\
B_{s1}&&B_{ss}\\ \hline
C_{11}&\dots&C_{1s}&J_{m_1}(\mu _1)&&0\\
\vdots&&\vdots&&\ddots\\
C_{r1}&\dots&C_{rs}&0&&J_{m_r}(\mu _r)},
\end{equation}
in which each $B_{ii}$ is a lower triangular matrix with the single eigenvalue $\lambda_i$, each $J_{m_j}(\mu _j)$ is of the form \eqref{zxc}, and
\begin{equation}\label{MJX}
C_{ij}=0\quad\text{ if }
\lambda _i\ne\mu _j.
\end{equation}
\end{lemma}

\begin{proof}
Let $A=\matt{B&0\\C&D}$ be a \kv matrix of $\mc A$ in an $m$-orthonormal basis. Changing the basis, we can reduce it by transformations $S^{-1}AS$, in which $S$ is an arbitrary nonsingular matrix of the form \eqref{fem}. By Schur's triangularization theorem \cite[Theorem 2.3.1]{hor}, there exists a unitary matrix $P$ such that $P^{-1}BP$ is lower triangular with eigenvalues in any prescribed order; we take them in the order
\[
\lambda_1,\dots,\lambda _1,
\lambda_2,\dots,\lambda _2,\dots,
\lambda_s,\dots,\lambda _s.
\]
There exists a nonsingular matrix $Q$ such that $Q^{-1}DQ$ is a Jordan matrix. Replacing $A$ by $(P\oplus Q)^{-1}A(P\oplus Q)$, we reduce $A$ to the form \eqref{fri}, in which the blocks $C_{ij}$ may not satisfy \eqref{MJX}.

If \eqref{MJX} does not hold, then we take the first $C_{pq}\ne 0$ with
$\lambda _p\ne\mu _q$
in the sequence
\begin{equation}\label{aah}
C_{1s},\ C_{1,s-1}, \ \dots, \ C_{11}, \ C_{2s}, \
C_{2,s-1}, \ \dots, \ C_{21}, \ \dots\ .
\end{equation}
Take $S:=\matt{I_m&0\\X&I}$, in which $X=[X_{ij}]$ is a block matrix partitioned conformally with $[C_{ij}]$, and  $X_{ij}=0$ if $(i,j)\ne (p,q)$.
Then all blocks of $A$ outside of $[C_{ij}]$ and all blocks $C_{ij}$ that precede $C_{pq}$ in \eqref{aah} coincide with the corresponding blocks of $A'=S^{-1}AS$. The block $C_{pq}$ becomes  $C'_{pq}=C_{pq}-X_{pq}B_{qq}+ J_{m_p}(\mu _p) X_{pq}$ in $A'$. Since
$\mu_p\ne\lambda_q$, by Gantmacher \cite[Chapter VIII, \S\,3]{gan1} there exists $X_{pq}$ such that $C'_{pq}=0$.

In the same manner, we take the first nonzero block $C'_{p'q'}$ with
$\lambda _{p'}\ne\mu _{q'}$
in the sequence
\[
C'_{1s},\ C'_{1,s-1}, \ \dots, \ C'_{11}, \ C'_{2s}, \
C'_{2,s-1}, \ \dots, \ C'_{21}, \ \dots,
\]
make it zero, and so on, until we obtain a matrix \eqref{fri} satisfying \eqref{MJX}.
\end{proof}

\begin{lemma}   \label{feo}
A  \kv matrix $A=\matt{B&0\\C&D}$ is semiunitarily similar to a  \kv matrix $A'=\matt{B'&0\\C'&D'}$ if and only if $A$ is reduced to $A'$ by a finite sequence of the following transformations:
\begin{itemize}
\item [\rm(i)] a unitary row-transformation in $[B\,0]$,
then the inverse unitary column-transformation in
$\matt{B\\C}$ (thus, $B$ is reduced by
transformations of unitary similarity);

\item [\rm(ii)] an elementary row-transformation in
$[C\,D]$, then the inverse
column-transformation in $\matt{0\\D}$ (thus, $D$ is reduced by transformations of similarity);

\item[\rm(iii)] adding  row $i$ of $B$ multiplied by $c\in\cc$
to row $j$ of $C$, then subtracting  column $j$
of $D$ multiplied by $c$ from column $i$ of $C$.
\end{itemize}
\end{lemma}

\begin{proof}
Each transformation (i), (ii), or (iii)  is a transformation $A\mapsto
R^{-1}AR$ with $R$ of the form
       \begin{equation}           \label{2.6}
\begin{bmatrix}
 R_{1}&0 \\ 0&I
\end{bmatrix},\quad
\begin{bmatrix}
 I&0 \\ 0&R_{2}
\end{bmatrix},\quad\text{or}\quad
\begin{bmatrix}
 I&0 \\ R_{3}&I
\end{bmatrix},
       \end{equation}
respectively, in which $R_{1}$ is a unitary matrix, $R_{2}$ is an elementary matrix, and $R_{3}$ is a matrix with only one nonzero entry. Each nonsingular matrix $S$ of the form \eqref{fem} is a product
of matrices of the form \eqref{2.6}, and so the transformation $A\mapsto S^{-1}AS$ is a sequence of transformations of types (i)--(iii).
\end{proof}

\section{Proof of Theorem \ref{t2.2}}\label{hhgy}

Let $\cal A$ be an indecomposable bounded
operator on a semiunitary space $U$ and $0\ne U_0\ne U$. Let $\lambda $ be an eigenvalue of $\mc A_0$. There is a semiorthonormal basis, in which the \kv matrix $A$ of $\mc A$ has the form  \eqref{fri} with $\mu_1=\lambda $.

\subsection{Proof of Theorem \ref{t2.2}(a)}

If $\lambda $ is not an eigenvalue of
${\cal A}_1$, then $C_{11},\dots,C_{1s}$ are zero by  Lemma \ref{htp}, in this case $J_{m_1}(\mu_1)$ is a block-direct summand. Since $\mc A$ is indecomposable, $A=J_{m_1}(\mu_1)$, hence $U=U_0$, which contradicts the condition  $0\ne U_0\ne U$.
Therefore, $\lambda $ is an eigenvalue of
${\cal A}_1$. By Lemma \ref{htp}, we can take $\lambda_{s}=\lambda $.

\subsection{Proof of Theorem \ref{t2.2}(b)}

Replacing $\mc A$ by $\mc A- \lambda \mathds{1}$, we make $\lambda =0$. Then $B_0:=B_{ss}$ has the eigenvalue $0$ and the other $B_{ii}$ are nonsingular. Using transformation (ii) from Lemma \ref{feo}, we reduce $A$  to the form
\begin{equation}\label{111a}
A=\begin{MAT}(@){4cc4c.cc.ccc.c1c4}
 \first4
B_2&      0&0&0&0&0&0&0&\dots&0\\
B_1&B_0&0&0&0&0&0&0&\dots&0\\4
E_{11}&F_{11}&0_{n_1}&&&&&&0\!\!\!\!&0\\.
E_{21}&F_{21}&&0_{n_2}&0&&&&&0\\
E_{22}&F_{22}&&I_{n_2}&0_{n_2}&&&&&0\\.
E_{31}&F_{31}&&&&0_{n_3}&0&0&&0\\
E_{32}&F_{32}&&&&I_{n_3}&0_{n_3}&0&&0\\
E_{33}&F_{33}&&&&0&I_{n_3}&0_{n_3}&&0\\.
\smash{\vdots}&\smash{\vdots}&0&&&
       &&&\smash{\ddots}&\smash{\vdots}\\1
*&*&0&0&0&0&0&0&\dots&K\\4
\end{MAT}
\end{equation}
in which $B_2$ and $K$ are nonsingular, $B_0$ has only the eigenvalue $0$, and $n_1,n_2,\ldots\in\{0,1,2,\dots\}$ (if $n_i=0$, then $J_i(0_{n_i})$ is absent).

If there exists  $F_{ij}\ne 0$ with $j\ge 2$, then let $F_{pq}$ be the first nonzero block in the sequence
\begin{equation}\label{bhn}
F_{22},\ F_{33},\ F_{32},\ F_{44},\ F_{43},\ F_{42},\ \dots\,.
\end{equation}
We make $F_{pq}=0$ by adding linear combinations of columns of $I_{n_p}$ to columns of $F_{pq}$ (transformation (iii) from Lemma \ref{feo}); these transformations and
the inverse row-transformations do not change the blocks that precede  $F_{pq}$ in \eqref{bhn}.
We repeat this reduction until we obtain $F_{i2}=F_{i3}=\dots=F_{ii}=0$ for all $i$.

If there exists  $E_{ij}\ne 0$, then let $E_{pq}$  be the first nonzero block in the sequence
\begin{equation}\label{bhd}
E_{11},\ E_{21},\ E_{22},\ E_{31},\ E_{32},\ E_{33},\ \dots\,.
\end{equation}
Adding linear combinations of rows of $B_2$ to rows of $E_{pq}$, we make $E_{pq}=0$; these transformations and the inverse transformations of columns do not change the blocks that precede  $E_{pq}$ in \eqref{bhd}.
We repeat this reduction until we make $E_{ij}=0$ for all $i,j$ and obtain
\begin{equation}\label{111}
A=\begin{MAT}(@){4cc4c.cc.ccc.c1c4}
 \first4
B_2&      0&0&0&0&0&0&0&\dots&0\\
B_1&B_0&0&0&0&0&0&0&\dots&0\\4
0&F_{11}&0_{n_1}&&&&&&0\!\!\!\!&0\\.
0&F_{21}&&0_{n_2}&0&&&&&0\\
0&0&&I_{n_2}&0_{n_2}&&&&&0\\.
0&F_{31}&&&&0_{n_3}&0&0&&0\\
0&0&&&&I_{n_3}&0_{n_3}&0&&0\\
0&0&&&&0&I_{n_3}&0_{n_3}&&0\\.
\smash{\vdots}&\smash{\vdots}&0&&&
       &&&\smash{\ddots}&\smash{\vdots}\\1
*&*&0&0&0&0&0&0&\dots&K\\4
\end{MAT}\,.
\end{equation}

A linear combination of rows of $F_{i1}$ with $i>1$ can be added to any row of $F_{11}$; the inverse transformation of columns do not change $A$. A linear combination of rows of $F_{31}$ can be added  to any row of $F_{21}$; the inverse transformation of columns spoils the second block to the right of $I_{n_2}$; we restore it adding linear combinations of rows of $I_{n_3}$; the inverse transformation of columns do not change $A$. In a similar way,
\begin{equation}\label{gfv}
\parbox[c]{0.84\textwidth}{a linear combination of rows of $F_{i1}$ can be added  to any row of $F_{j1}$ with $j<i$. }
\end{equation}

Let us
consider the matrix
\begin{equation}\label{wxo}
F:=\mat{F_{11}\\F_{21}\\\vdots\\F_{t1}}
\end{equation}
formed by all blocks $F_{i1}$. Let $p\times q$ be the size of $F$.
If its rows are linearly dependent, then we make a zero row by transformations \eqref{gfv}, which is impossible since $\cal A$ is indecomposable and $U_0\ne U$. Hence, the rows of $F$ are linearly independent, and so $p\le q$. However, $p$ is the geometric multiplicity of the eigenvalue $0$ in ${\cal A}_0$ and $q$ is the algebraic multiplicity of $0$ in
${\cal A}_1$, which proves Theorem \ref{t2.2}(b).

\subsection{Proof of Theorem \ref{t2.2}(c)}

Let $\mc A$ be given by the matrix
\eqref{111a}, in which $B_2$ and $K$ are nonsingular and $B_0$ has only the eigenvalue zero. By transformations of unitary similarity, we reduce $B_0$ to the form $[B_3\,0]$, in which the columns of $B_3$ are linearly independent. Then we join to $K$ all blocks $J_i(0_{n_i})$ with the exception of one block $J_k(0_{n_k})$
with $n_k\ne 0$ and obtain
\begin{equation*}\label{133}
A=\begin{MAT}(@){4ccc4ccccc1c4}
 \first4
B_2&0&0 & 0&0&0&\dots&0&0\\
B_1&B_3&0 & 0&0&0&\dots&0&0\\4
E_1&G_1&H_1& 0_{n_k}&&&&0&0\\
E_2&G_2&H_2& I_{n_k}&0_{n_k}&&&&0\\
E_3&G_3&H_3& &I_{n_k}&0_{n_k}&&&0\\
\smash{\vdots}&\smash{\vdots}&\smash{\vdots}& &&\smash{\ddots}&\smash{\ddots}&\phantom{0}&\smash{\vdots}\\
E_k&G_k&H_k&0 &&&I_{n_k}&0_{n_k}&0\\1
*&*&*&0 &0&0&\dots&0&M\\4
\end{MAT}
\end{equation*}
in which $M$ does not contain the direct summands $J_k(0)$.

We make zero all blocks $E_i$ and $G_i$ by transformations (iii) from Lemma \ref{feo} as follows.
If the first nonzero block in the sequence
\begin{equation}\label{bhi}
G_{1},\ E_{1},\ G_{2},\ E_{2},\ G_{3},\ E_{3},\ \dots
\end{equation}
is $G_i$, then we make $G_i=0$ by adding linear
combinations of rows that cross $B_3$. If the first nonzero block in \eqref{bhi} is $E_i$, then we made $E_i=0$ by adding linear combinations of rows that cross $B_2$. All preceding blocks in  \eqref{bhi} remain zero under these transformations and the inverse transformations.

We reduce $H_1,H_2,\dots,H_k$ by transformations (ii)  from Lemma \ref{feo} as follows. If
\[
R:=\mat{R_1&&&0\\R_2&R_1
\\[-2mm]\ddots&\ddots&\ddots\\[-2mm]
R_k&\ddots&R_2&R_1},
\]
in which all blocks are $n_k\times n_k$ and $R_1$ is nonsingular, then $RJ_k(0_{n_k})R^{-1}=J_k(0_{n_k})$, and so we can reduce
\[
H:=\mat{H_1\\[-2mm]\vdots\\H_k}
\]
by transformations $RH$. Therefore, we can reduce $H$ by the following transformations:
\begin{itemize}
  \item[($\alpha$)]  Simultaneous elementary transformations of rows in all $H_i$.

  \item[($\beta$)] For any $c\in\cc$, $q\ne 0$,  $i$, and $j$, adding row $i$ of $H_p$ multiplied by $c$ to row $j$ of $H_{p+q}$ simultaneously for all $p=1,2,\dots, k-q$.

\item[($\gamma$)] Elementary transformations of columns of $H$.
\end{itemize}

We make
\[
H_1=\mat{I_{r_1}&0\\0&0}
\]
by transformations ($\alpha$) and ($\gamma $).

We
partition $H_2=\matt{H_{21}\\H_{22}}$ so that $H_{21}$ has $r_1$ rows, make zeros in $H_{22}$ under $I_{r_1}$, reduce the remaining part of $H_{22}$ to $I_{r_2}\oplus 0$, and obtain
\[
H_2=\left[\begin{MAT}{c}
  H_{21}\\.
\begin{MAT}{c.c}
  0_{n_k-r_1,r_1}&
\ma{I_{r_2}&0\\0&0\\[-2mm]}
  \\
\end{MAT}
  \\
\end{MAT}\right],\quad H_{21}\text{ is }r_1\times h,
\]
in which $h$ is the number of columns in $H$.

We partition  $H_3=\matt{H_{31}\\H_{32}}$ so that $H_{31}$ has $r_1+r_2$ rows, make zeros in $H_{32}$ under $I_{r_1}$ and $I_{r_2}$, reduce the remaining part of $H_{22}$ to $I_{r_3}\oplus 0$, and obtain
\[
H_3=\left[\begin{MAT}{c}
  H_{31}\\.
\begin{MAT}{c.c}
  0_{n_k-r_1-r_2,r_1+r_2}&
\ma{I_{r_3}&0\\0&0\\[-2mm]}
  \\
\end{MAT}
  \\
\end{MAT}\right],\quad H_{31}\text{ is }(r_1+r_2)\times h.
\]

We repeat this reduction
until we obtain
\[
H_k=\left[\begin{MAT}{c}
  H_{k1}\\.
\begin{MAT}{c.c}
  0_{n_k-r_1-\dots-r_{k-1},r_1+\dots+r_{k-1}}&
\ma{I_{r_k}&0\\0&0\\[-2mm]}
  \\
\end{MAT}
  \\
\end{MAT}\right],\quad H_{k1}\text{ is }(r_1+\dots+r_{k-1})\times h.
\]

If the last row of $H_k$ is zero, then the last row is
zero in all $H_i$, which contradicts the indecomposability of
$\cal A$ and $U_0\ne U$. Hence  $n_k=r_1+\dots+r_k$, which is the
number of blocks
$J_k(0)$ in the Jordan canonical form of $\mc A_0$. This proves Theorem \ref{t2.2}(c) since $n_k\le h$ and $h$ is the geometric multiplicity of 0 in $\mc A_1$.

\section{Proof of Theorem \ref{ttt}}\label{v2v}

\subsection{Proof of Theorem \ref{ttt}(a)}

It suffices to show that for each indecomposable seminormal operator there exists an orthonormal basis in which its \kv matrix  has the form \eqref{3.1}; the uniqueness of the block-direct sum follows from \eqref{frb}.

Let  ${\cal A}$ be an indecomposable seminormal operator on a semiunitary space $U$. There is a semiorthonormal basis, in which its \kv matrix has the form
\eqref{fri}. Since ${\cal A}_1$ is normal, we can take all $B_{ii}={\lambda}_iI$ and $B_{ij}=0$
if $i\ne j$. Since $\cal A$ is indecomposable and \eqref{MJX} holds, $A$ has a single eigenvalue $\lambda $.

If $U_0=0$, then $A=[\lambda]_1$. If
$U_0=U$, then $A=[J_n(\lambda)]_0$.

Let $0\ne U_0\ne U$. It is sufficient to consider the case  $\lambda =0$.  There is a semiorthonormal basis, in which the \kv matrix of $\mc A$ has the form \eqref{111} with $B_0=0$ and without $B_2$ and $K$.
Its submatrix \eqref{wxo} has linearly independent rows; it is reduced by elementary column-transformations, elementary row-transformations within each $F_{i1}$, and  the transformations \eqref{gfv}. Using these transformations, we reduce $F_{t1}$ to the form
$[I\,0\,\dots,\,0]$, and then make zero all the entries of $F$ over $I$. Since $\mc A$ is indecomposable, $F=F_{t1}=[1]$, and so $A=[J_t(0)]_1$.

\subsection{Proof of  Theorem \ref{ttt}(b)}

Let $\mc F:U\times U\to \cc$ be a sesquilinear form on a semiunitary space $U$. Let $F=\matt{A&B\\C&D}$ be its \kv matrix in a semiorthonormal basis. Changing the basis, we can reduce $F$ by transformations of \emph{semiunitary *congruence} $S^*FS$, in which $S$ is a nonsingular matrix of the form \eqref{fem}.

\begin{lemma}   \label{l3.2}
A  \kv matrix $F=\matt{A&B\\C&D}$ is semiunitarily *congruent  to a  \kv matrix $F'=\matt{A'&B'\\C'&D'}$ if and only if $F$ is reduced to $F'$ by a sequence of the following transformations:
\begin{itemize}
\item [\rm(i)] a unitary row-transformation in $[A\,B]$,
then the *congruent column-transformation in
$\matt{A\\C}$ (thus, $A$ is reduced by
transformations of unitary *congruence);

\item [\rm(ii)] an elementary row-transformation in
$[C\,D]$, then the *congruent
column-transformation in $\matt{B\\D}$ (thus, $D$ is reduced by transformations of *congruence);

\item[\rm(iii)] adding  row $i$ of $[C\,D]$ multiplied by $c\in\cc$
to row $j$ of $[A\,B]$, then adding  column $i$
of $\matt{B\\D}$ multiplied by $\bar c$ to column $j$ of $\matt{A\\C}$.
\end{itemize}
\end{lemma}

\begin{proof}
This lemma follows from the fact that each nonsingular matrix $S$ of the form \eqref{fem} is a product
of matrices of the form \eqref{2.6}.
\end{proof}

A \kv matrix is \emph{indecomposable for semiunitary *congruence} if it is not semiunitarily  *congruent to a block-direct sum of \kv matrices of smaller sizes. By analogy with the proof of \eqref{frb}, we apply \cite[Theorem 2]{ser_izv} to a pair consisting of the Hermitian form $\l \cdot,\cdot\r$ and a sesquilinear form and obtain that
\begin{equation}\label{mhm}
\parbox[c]{0.84\textwidth}
{each \kv matrix  is semiunitarily *congruent to a block-direct sum of  \kv matrices that are indecomposable for semiunitary *congruence, and this sum is uniquely determined, up to semiunitary *congruence of summands.}
\end{equation}
Let $F=\matt{A&B\\B^*&D}$, in which  $A^*=A$  and $D^*=D$,
be a Hermitian \kv matrix that is indecomposable with respect to semiunitary *congruence. Due to \eqref{mhm}, it suffices to prove that $F$ is semiunitarily *congruent to exactly one \kv matrix of the form \eqref{3.2}.

We reduce $D$ to the form $I\oplus
(-I)\oplus 0$ by transformations (ii) from Lemma \ref{l3.2}.
Make zero the entries of $B$ over $I$ and $-I$ by transformations (iii).

If $D\ne 0$, then $F$ is $[1]_0$ or $[-1]_0$ since
$F$ is indecomposable for semiunitary *congruence.

Let $D=0$. We reduce $B$ to the form $I\oplus 0$ by transformations (i) and (ii), partition $A$, and obtain
\[
F=\arrr{A}{B}{B^*}{D}=\left[
\begin{array}{cc|cc}
X&Y&I&0\\Y^*&Z&0&0\\\hline I&0&0&0\\0&0&0&0\end{array} \right],\quad
\text{in which }X^*=X,\ Z^*=Z.
\]
Make $Y=0$ and $X=0$ by additions of linear combinations of columns of $\matt{B\\0}$ (transformations (iii)). Then reduce $Z$ to a real diagonal matrix by transformations (i). Since $F$ is indecomposable for semiunitary *congruence, it is $\matt{0&1\\1&0}_1$, $[\lambda]_1$ $(\lambda\in{\mathbb R})$, or $[0]_0$,
which proves Theorem \ref{ttt}(d).

\section{An algorithm for reducing to canonical form of matrices of systems of linear mappings on semiunitary
spaces}
\label{z3v}

\subsection{Semiunitary representations of quivers}\label{n4d}

Let $\un m=(m_1,\dots,m_l)$ and
$\un{n}=(n_1,\dots,n_r)$ be two sequences of \emph{nonnegative} integers. An \emph{$\un{m}\times\un{n}$
matrix} is a block matrix $M=[M_{ij}]_{i=1}^l{}_{j=1}^r$, in which every block $M_{ij}$ is $m_i\times n_j$ (thus, some horizontal and vertical strips can be empty).

A \emph{quiver} is a directed graph, in which loops and multiple edges are allowed. Its  \emph{semiunitary representation} is given by assigning a  semiunitary space to each vertex and a linear mapping of the corresponding spaces to each arrow.

\begin{example}\label{vtj}
Let us consider the quiver $Q$ and its semiunitary representation $\mc R$:
\begin{equation}\label{kshd}
\begin{split}
Q:\hspace{-30pt}
\xymatrix@R=2pc@C=1pc{
 &{u}&\\
 {v}\ar@(ul,dl)@{->}_{\gamma }
 \ar@{->}[ur]^{\alpha}
  \ar@/^/@{->}[rr]^{\delta }
 \ar@/_/@{->}[rr]_{\varepsilon} &&{w}
 \ar@{->}[ul]_{\beta}
 \ar@{->}@(u,ur)^{\zeta}
 \ar@{<-}@(dr,r)_{\eta}
 }
         %%%%%%%%%%%%%%%
\qquad\qquad
\mc R:\hspace{-30pt}
\xymatrix@R=2pc@C=1pc{
 &{U}&\\
  \save
!<-1.5mm,0cm>
 {V}\ar@(ul,dl)@{->}_{\mc C}
 \restore
 \ar@{->}[ur]^{\mc A}
  \ar@/^/@{->}[rr]^{\mc D}
 \ar@/_/@{->}[rr]_{\mc E} &&{W}
 \ar@{->}[ul]_{\mc B}
 \ar@{->}@(u,ur)^{\mc F}
 \ar@{<-}@(dr,r)_{\mc G}
 }
\end{split}
\end{equation}
in which $U,V,W$ are semiunitary spaces and $\cal A,B,\dots,G$ are linear mappings. Choosing semiorthonormal bases in the spaces $U,V,W$, we  give these mappings by their matrices $A,B,\dots,G$. Changing the bases, we can reduce them by transformations
\begin{equation}\label{ksd2}
\begin{split}
{\xymatrix@=3pc{
 &{u}&\\
   \save
!<0mm,0cm>
 {v}\ar@(ul,dl)
 @{->}_{S_v^{-1}CS_v}
 \restore
 \ar@{->}[ur]^{ S_u^{-1}AS_v}
  \ar@/^/@{->}[rr]^{S_w^{-1}DS_v}
 \ar@/_/@{->}[rr]_{ S_w^{-1}ES_v} &&{w}
 \ar@{->}[ul]_{ S_u^{-1}BS_w}
 \ar@{->}@(u,ur)^(.7){S_w^{-1}FS_w}
 \ar@{<-}@(dr,r)_{S_w^{-1}GS_w}
 }}\end{split}
\end{equation}
where $S_u,S_v,S_w$ are the transition matrices of the form \eqref{fem}. Therefore, the problem of classifying semiunitary representations of the quiver $Q$ is the problem of classifying matrix tuples  $(A,B,\dots,G)$ up to transformations \eqref{ksd2} given by nonsingular matrices of the form \eqref{fem}.
\end{example}

\subsection{Transformations of $(\Lambda,\mc U)$-similarity}

Recall that a \emph{matrix algebra} is a vector space of matrices of the same size that contains the identity matrix and is closed
with respect to multiplication.

The following definition of a reduced algebra is given in \cite[Definition 1.1]{ser_can} over an arbitrary field. The term ``reduced'' is used because  for  every  algebra $\Xi\subset\cc^{n\times n}$ there  exists  a  nonsingular
matrix $P\in\cc^{n\times n}$ such that $P^{-1}\Xi P$ is a reduced matrix algebra (see \cite[Theorem 1.1]{ser_can}).

\begin{definition} \label{d1.1}
{\rm(a)} A \emph{basic reduced algebra} $\Gamma$ over $\cc$ is an algebra of $t\times t$ upper triangular complex matrices satisfying the following condition:
\begin{itemize}
\item there exists an equivalence relation
\begin{equation}  \label{0}
\approx\ \ \text{in } T=\{1,\dots,t\},
\end{equation}

\item for each pair ${\cal I,J}\in
T/\!\approx$ of equivalence
classes, there exists a system of linear equations
\begin{equation}       \label{00}
\sum_{{\cal I} \ni i < j \in {\cal J}}
c_{ij}^{(1)}x_{ij} = 0,\
\sum_{{\cal I} \ni i < j \in {\cal J}}
c_{ij}^{(2)}x_{ij} = 0,\
\dots,\
\sum_{{\cal I} \ni i < j \in {\cal J}}
c_{ij}^{(r_{_{\cal IJ}})}x_{ij} = 0,
\end{equation}
in which all $c_{ij}^{(l)} \in \cc$
and $r_{_{\cal IJ}} \geq 0$,
\end{itemize}
such that $\Gamma$ consists of all upper triangular
complex matrices
\begin{equation*}         \label{1zz}
   S=\mat{s_{11}&s_{12}&\cdots& s_{1t}\\
&s_{22}&\ddots&\vdots\\ &&\ddots&s_{t-1,t}\\
 0 &&& s_{tt}},
\end{equation*}
in which all diagonal entries satisfy the condition
\begin{equation*}         \label{2zz}
s_{ii} = s_{jj} \qquad \text{if } i \approx
j\,,
\end{equation*}
and all entries above the diagonal satisfy the system of equalities
 \begin{equation}         \label{3zz}
\sum_{{\cal I} \ni i < j \in {\cal J}}
c_{ij}^{(1)}s_{ij} = 0,
\sum_{{\cal I} \ni i < j \in {\cal J}}
c_{ij}^{(2)}s_{ij} = 0,\
\dots,
\sum_{{\cal I} \ni i < j \in {\cal J}}
c_{ij}^{(r_{_{\cal IJ}})}s_{ij} = 0
\end{equation}
for each pair ${\cal I,J}\in
T/\!\approx$.

{\rm(b)} Let $\Gamma$ be a basic reduced algebra over $\cc$, and let $\un n=(n_1,\dots,n_t)$ be a sequence of nonnegative integers such that $n_{i} = n_{j}$ if $i \approx
j$. A \emph{reduced
algebra $\Lambda:=\Gamma^{\un n \times\un n}$} is the algebra of $\un n \times\un n$ matrices that consists of all
upper block triangular
$\underline{n}\times\underline{n}$
matrices
\begin{equation}         \label{1}
   S=\begin{bmatrix}
S_{11}&S_{12}&\cdots& S_{1t}\\
&S_{22}&\ddots&\vdots\\ &&\ddots&S_{t-1,t}\\
 0 &&& S_{tt}  \end{bmatrix},\qquad\text{each }
   S_{ij}\text{ is } {n_i\times n_j},
\end{equation}
in which all diagonal blocks satisfy the
condition
\begin{equation}         \label{2}
S_{ii} = S_{jj} \qquad \text{if } i \approx
j\,,
\end{equation}
and  all blocks above the diagonal satisfy the system of equalities
 \begin{equation}         \label{3}
\sum_{{\cal I} \ni i < j \in {\cal J}}
c_{ij}^{(1)}S_{ij} = 0,
\sum_{{\cal I} \ni i < j \in {\cal J}}
c_{ij}^{(2)}S_{ij} = 0,\
\dots,
\sum_{{\cal I} \ni i < j \in {\cal J}}
c_{ij}^{(r_{_{\cal IJ}})}S_{ij} = 0
\end{equation}
for each pair ${\cal I,J}\in
T/\!\approx$.
\end{definition}

Thus, $\Lambda$ as a complex vector space is a direct sum of vector spaces
\begin{equation}\label{omo}
\Lambda=\Delta\oplus\Upsilon,
\end{equation}
in which $\Delta$ consists of block diagonal martices $S_{11}\oplus S_{22}\oplus\dots \oplus S_{tt}$ and $\Upsilon$
consists of block triangular matrices with zero block diagonal.

Let $\Lambda$ be a reduced algebra of $\un n\times\un n$ matrices, and let ${\cal U} \subset T$ be closed under the relation $\approx$ from \eqref{0}; that is, $i\in{\cal U}$ and $i\approx j$ imply $j\in{\cal U}$. Denote
by $\mc G(\Lambda, \mc U)$ the group of all nonsingular ${\un n\times\un n}$
matrices $S\in \Lambda$  satisfying the following condition:
\begin{equation*}                  \label{5.4a}
S_{ii}\quad \emph{is unitary if } i\in{\cal U}.
\end{equation*}

If all $n_i$ in $\un n=(n_1,\dots,n_t)$ are nonzero, then
\begin{equation*}
\label{bbv}
\parbox[c]{0.84\textwidth}{the algebra $\Lambda$, its partition $\un n\times\un n$, the equivalence relation \eqref{0}, and the set $\mc U$ are uniquely determined by the group $\mc G(\Lambda, \mc U)$ due to the structure of its matrices.
}
\end{equation*}
We say that two matrices $M$ and $N$ are
\emph{$(\Lambda,\mc U)$-similar} and write $M \sim_{(\Lambda,\mc U)} N$ if
there exists $S \in \mathscr G(\Lambda,\mc U )$ such that $S^{-1}MS = N$.

\subsection{The problem of classifying matrices up to $(\Lambda,\mc U)$-similarity contains the problem of classifying semiunitary representations of quivers}

A classification problem $\mc M$ \emph{contains} a classification problem $\mc N$ if a solution of $\mc N$ can be obtained from a solution of $\mc M$ (a formal definition is given in \cite[p.\,205]{bel-ser_comp}).

\begin{lemma}\label{llx}
For each quiver $Q$, there exist a basic reduced algebra $\Gamma$ and a set ${\cal U} \subset T$ closed under the equivalence \eqref{0} such that for each $\un n$ the problem of classifying ${\un n\times\un n}$ matrices up to $(\Gamma^{\un n\times\un n},\mc U)$-similarity contains the problem of classifying semiunitary representations of $Q$.
\end{lemma}

\begin{proof}
For the sake of clarify, we prove the lemma for the quiver $Q$ in \eqref{kshd}. Let us choose  semiorthonormal bases in the spaces $U,V,W$ and interchange the basis vectors such that the matrices of scalar products are given by the matrices
\begin{equation}\label{ect}
H_u=\mat{0_{m_1}&0\\0&I_{m_2}},\quad
H_v=\mat{0_{n_1}&0\\0&I_{n_2}},\quad
H_w=\mat{0_{r_1}&0\\0&I_{r_2}}.
\end{equation}
Let
$A,B,\dots,G$ be the matrices of a semiunitary representation $\mc R$ from \eqref{kshd}.
The transition matrices that preserve \eqref{ect} have the form
\begin{equation*}\label{ect1}
S_u=\mat{U_1&U_2\\0&U_3},\quad
S_v=\mat{V_1&V_2\\0&V_3},\quad
S_w=\mat{W_1&W_2\\0&W_3},
\end{equation*}
in which $U_3,V_3,W_3$ are unitary matrices.
Let us construct by
$A,B,\dots,G$ the matrix
\[
M:=\mat{0_{m_1+m_2}&A&B&0
\\
0&C&0&0
\\
0&D&F&0
\\
0&E&G&0_{r_1+r_2}}.
\]

 If
$
S:=S_u\oplus S_v\oplus S_w\oplus S_w,
$
then
\begin{equation*}\label{wsr}
S^{-1}MS=
\mat{0&S_u^{-1}AS_v&S_u^{-1}BS_w&0
\\
0&S_v^{-1}CS_v&0&0
\\
0&S_w^{-1}DS_v&S_w^{-1}FS_w&0
\\
0&S_w^{-1}ES_v&S_w^{-1}GS_w&0};
\end{equation*}
therefore, the blocks $A,B,\dots,G$ are reduced as in \eqref{ksd2}.

The matrices of the form
\[
S=\mat{U_1&U_2\\0&U_3}\oplus
\mat{V_1&V_2\\0&V_3}\oplus
\mat{W_1&W_2\\0&W_3}\oplus
\mat{W_1&W_2\\0&W_3}
\]
compose the group $\mathscr G(\Gamma^{\un n\times\un n},\mc U )$, and so
\begin{itemize}
  \item $T=\{1,2,\dots,8\}$, in which $5\approx 7$ and $6\approx 8$;
\item the system of equations \eqref{3zz} consists of the equations \[s_{ij}=0,\qquad 1\le i<j\le 8,\quad(i,j)\notin
      \{(1,2),(3,4),(5,6),(7,8)\}\]
    and $s_{56}=s_{78}$;
\item  $\mc U=\{2,4,6,8\}$;
  \item $\un n:=(m_1,m_2,n_1,n_2,r_1,r_2,r_1,r_2)$.
\end{itemize}
\vskip-2em
\end{proof}

\subsection{An algorithm for reducing matrices to canonical form under $(\Lambda,\mc U)$-similarity}        \label{sec3}

In this section, we give an algorithm
for reducing an $\underline{n}\times \underline{n}$ matrix $M$
to a canonical form  $M_{\text{can}}$ under
$(\Lambda,U)$-similarity.

A brief sketch of the algorithm is as follows.
We start from a triple
$(M,\Lambda,\mc U)$ that  consists of an $\un n\times \un n$ matrix $M$, a reduced algebra $\Lambda$ of $\un n\times \un n$ matrices, and a set ${\cal U} \subset T$ closed under $\approx$.
\begin{itemize}
  \item First we order the blocks of $M$ as follows:
\begin{equation}      \label{5}
M_{t1}<M_{t2}<\dots<M_{tt}<M_{t-1,1}<
M_{t-1,2}<\dots<M_{t-1,t}<\cdots.
\end{equation}

  \item Take the first block $M_{pq}$ that is changed by admissible transformations and reduce it to its canonical form $M_{pq}'$, which we describe below.
  \item Restrict the set of admissible transformations to those that preserve $M_{pq}'$ and refine the partition in accordance with the partition of $M_{pq}'$.
\end{itemize}
We obtain a new triple $(M',\Lambda',\mc U')$, in which $M'$ is $(\Lambda,\mc U)$-similar to $M$ and $\Lambda'
\varsubsetneqq\Lambda$.
We repeat this construction until we obtain the sequence
\begin{equation}\label{ksj}
(M,\Lambda,\mc U),\quad
(M',\Lambda',\mc U'),\quad
(M'',\Lambda'',\mc U''),\ \dots,\ (M^{(k)},\Lambda^{(k)},\mc U^{(k)}),
\end{equation}
in which $k\ge 0$ and $M^{(k)}$ is $(\Lambda^{(k)},\mc U^{(k)})$-similar only to itself. The obtaining matrix $M_{\text{can}}:=M^{(k)}$ is the \emph{canonical form of $M$ under $(\Lambda,\mc U)$-similarity} since we prove in Theorem \ref{t5.2} that
\begin{equation*}\label{bth}
M \sim_{(\Lambda,\mc U)} M_{\text{can}}\, ,\qquad\quad  M \sim_{(\Lambda,\mc U)} N\quad \Longleftrightarrow\quad M_{\text{can}}=N_{\text{can}}\, .
\end{equation*}

Let us construct $(M',\Lambda',\mc U')$. We say that a block $M_{ij}$ of $M$ is  {\it stable} if it is not changed by $(\Lambda,U)$-similarity transformations.
Let $M_{ij}$  be stable.
\begin{itemize}
  \item If $i\approx j$, then $M_{ij}=a_{ij}I$ for some $a_{ij}\in\cc$
since $M_{ij}$ is not changed by transformations $S_{ii}^{-1}M_{ij}S_{jj}$, in which $S_{ii}=S_{jj}$ is an arbitrary nonsingular or unitary matrix.
  \item If $i\napprox j$, then $M_{ij}=0$ (we
put $a_{ij}=0$) since it is not changed by transformations $S_{ii}^{-1}M_{ij}S_{jj}$, in which $S_{ii}$ and $S_{jj}$ are arbitrary nonsingular or unitary matrices.
\end{itemize}

If all blocks of $M$ are stable,
then $M$ is invariant under
$(\Lambda,U)$-similarity; we put $M_{\text{can}}:=M$.
Suppose that not all blocks of $M$ are stable.

Let $M_{pq}$ be the first nonstable block with respect to the  ordering \eqref{5},
and let $p \in
{\cal P}$ and $q \in {\cal Q}$, where $\mc P, \mc Q\in T/\!\approx$ are
equivalence classes.
If $M':=S^{-1}MS$, in which
$S\in \mathscr G(\Lambda, U)$ has the form
(\ref{1}), then the $(p,q)$ block of
the matrix $MS=SM'$ is
\[
M_{p1}S_{1q}+M_{p2}S_{2q}+
\dots+M_{pq}S_{qq}
=S_{pp}M'_{pq}+S_{p,p+1}M'_{p+1,q}+
\dots+S_{pt}M'_{tq}.
\]
Since all  $M_{ij}<M_{pq}$ are stable, we have that
\begin{equation}      \label{6}
a_{p1}S_{1q}+\dots+a_{p,q-1}S_{q-1,q}+
M_{pq}S_{qq}=S_{pp}M'_{pq}+
S_{p,p+1}a_{p+1,q}+\dots+S_{pt}a_{tq},
\end{equation}
in which we remove all summands with
$a_{ij}=0$; their sizes may differ from the size
of $M_{pq}$.

In the following cases, we
reduce $M_{pq}$ to some simple form $M_{pq}'$ and restrict the set of $(\Lambda,U)$-similarity
transformations to those that preserve $M_{pq}'$. Thus, we take
\begin{equation}        \label{8'a}
\mathscr G(\Lambda', U'):=\{S\in\mathscr G(\Lambda, U)\,|\,(M'S)_{pq}= (SM')_{pq}\},
\end{equation}
where $(X)_{pq}$ denotes the $(p,q)$ block of $X$.

\begin{description}
\item[\it Case I:] \emph{ $(ME)_{pq}\ne (EM)_{pq}$ for some $E\in \Upsilon$} (see \eqref{omo}). This means that the
    $r_{\!_{\mc P \mc Q}}$ equalities
    \eqref{3} with $(\mc I, \mc J)=(\mc P, \mc Q)$ do not imply
\begin{equation}      \label{7}
   a_{p1}S_{1q}+a_{p2}S_{2q}+
   \dots+a_{p,q-1}S_{q-1,q}
   =S_{p,p+1}a_{p+1,q}+\dots+
   S_{pt}a_{tq}
\end{equation}
(see \eqref{6}).
Take $S \in
\mathscr G(\Lambda, U)$ of the form \eqref{1} with $(S_{11},S_{22},\dots,S_{tt})=(I,\dots,I)$ that satisfies  \eqref{6} with
$M'_{pq}=0$. We obtain $M'=S^{-1}MS$, in which $M'_{pq}=0$. The reduced algebra $\Lambda'$ defined in \eqref{8'a} consists of all
$S\in \Lambda$ satisfying \eqref{7} (we
add it to the system \eqref{3}), its decomposition \eqref{omo} has the form
$\Lambda'=\Delta\oplus\Upsilon'$ with the same $\Delta$.
We have $\un n'=\un n$, and $\mc U'=\mc U$.

\item[\it Case II:] \emph{$(ME)_{pq}= (EM)_{pq}$  for all $E\in \Upsilon$}. This means that the
    $r_{\!_{\mc P \mc Q}}$ equalities
    \eqref{3} with $(\mc I, \mc J)=(\mc P, \mc Q)$
imply \eqref{7}. Then \eqref{6}
    simplifies to $M_{pq}S_{qq}= S_{pp}M'_{pq}$; that is,
\begin{equation}      \label{8}
  M'_{pq}=S_{pp}^{-1} M_{pq}S_{qq}.
\end{equation}
In each of the following subcases, we reduce $M_{pq}$ to some simple form. In the sequel,  we use only those $S_{pp}$ and $S_{qq}$ that preserve this form. Thus,
\begin{equation*}\label{zjo}
\Lambda'=\Delta'\oplus\Upsilon,
\end{equation*}
in which $\Upsilon$ is the same as in \eqref{omo}.

\item[\it Subcase II(a):]
    {\it $p\napprox q$ and $p,q\notin \mc U$.} Then $S_{pp}$ and $S_{qq}$ in \eqref{8} are
arbitrary nonsingular matrices.
We
choose $S\in \mathscr G(\Lambda, U)$ such that
\begin{equation}\label{wses}
M'_{pq}= S_{pp}^{-1}M_{pq}S_{qq}=
      \mat{0&I\\0&0}.
\end{equation}
The algebra $\Lambda'$ consists of all
$S\in \Lambda$ for which
\begin{equation}\label{uhuh}
  S_{pp}\mat{0&I\\0&0} =\mat{0&I\\0&0}S_{qq};
\end{equation}
that is, for which $S_{pp}$ and $S_{qq}$ have the form
\begin{equation}\label{vgy}
S_{pp}= \mat{P_1&P_2\\0&P_3} ,\quad
S_{qq}= \mat{
Q_1& Q_2\\0& Q_3}, \qquad
P_1 = Q_3.
\end{equation}
We make the additional partition of the matrices of $\Lambda $ extending the partitions of $S_{pp}$ and $S_{qq}$, and rewrite the equalities
\eqref{2} and \eqref{3} for smaller blocks.

\item[\it Subcase II(b):]
    {\it $p\napprox q$, $p\in \mc U$, and $q\notin \mc U$.} Then $S_{pp}$ is unitary and $S_{qq}$ is nonsingular. We reduce $M_{pq}$ to the form \eqref{wses}. The equality \eqref{uhuh} with unitary $S_{pp}$ implies \eqref{vgy} with $P_2=0$ and unitary $P_1$ and $P_3$. Hence, $\Lambda'$ consists of all
$S\in \Lambda$, in which $S_{pp}$ and $S_{qq}$ are of the form \eqref{vgy} with $P_2=0$; $P_1$ and $P_3$ are unitary in the matrices of $\mathscr G(\Lambda', U')$.

\item[\it Subcase II(c):]
    {\it $p\napprox q$, $p\notin \mc U$, and $q\in \mc U$.} Then $S_{pp}$ is nonsingular and $S_{qq}$ is unitary. We reduce $M_{pq}$ to the form \eqref{wses}. The equality \eqref{uhuh} with unitary $S_{qq}$ implies \eqref{vgy} with $Q_2=0$ and unitary $Q_1$ and $Q_3$.

\item[\it Subcase II(d):]
    {\it $p\napprox q$ and $p,q\in \mc U$.} Then $S_{pp}$ and $S_{qq}$ are unitary. We reduce $M_{pq}$ by transformations \eqref{8} to the form
\begin{equation*}       \label{5.9a}
M_{pq}'=a_1I\oplus\dots\oplus a_kI\oplus 0\quad
\text{with real } a_1>\dots>a_{k}>0.
 \end{equation*}
If $S_{pp}M_{pq}'=M_{pq}'S_{qq}$ with unitary $S_{pp}$ and $S_{qq}$, then they have the form \[
S_{pp}=S_1\oplus\dots\oplus S_{k} \oplus P,\qquad S_{qq}=S_1\oplus\dots\oplus S_{k}\oplus Q,\] in which each $S_i$
has the same size as $a_iI$
(see \cite[Lemma 2.1(a)]{ser2}).

\item[\it Subcase II(e):]
{\it $p\approx q$ and $p,q\notin \mc U$.} Then $S_{pp}=S_{qq}$ is an arbitrary nonsingular matrix. We reduce $M_{pq}$  by transformations \eqref{8} to its Weyr canonical form under similarity
\[
M_{pq}'=
\mat{\lambda_1I&\matt{I\\0}&&0\\
&\lambda_1I&\ddots&\\
&&\ddots&\matt{I\\0}\\
0&&&\lambda_1I}\oplus\dots\oplus \mat{\lambda_kI&\matt{I\\0}&&0\\
&\lambda_kI&\ddots&\\
&&\ddots&\matt{I\\0}\\
0&&&\lambda_kI},
\]
in which $\lambda_1\prec\dots\prec \lambda_k$ with respect to the lexicographical ordering of complex numbers:
\begin{equation}\label{drc}
a+bi\prec c+di\qquad \text{if either $a<c$, or $a=c$ and $b<d$.}
\end{equation}
All commuting with $M_{pq}'$  matrices
form a reduced algebra with equations \eqref{00} of the
form $x_{ij}=x_{i'j'}$ and $x_{ij}=0$ (see \cite[Section 1.3]{ser_can}).

\item[\it Subcase II(f):]
{\it $p\approx q$ and $p,q\in \mc U$.} Then $S_{pp}=S_{qq}$ is a unitary matrix. We reduce $M_{pq}$ by transformations \eqref{8}  to the form
\begin{equation}         \label{5.10r}
M_{pq}'=\mat{\lambda_1I&F_{12}&
\dots&F_{1k}\\ &\lambda_2I &\ddots&\vdots \\
&&\ddots &F_{k-1,k} \\ 0&&&\lambda_kI},\ \ \
      \parbox{140pt}
{$\lambda_1
\preccurlyeq\dots\preccurlyeq\lambda_k$
(see \eqref{drc});\\ if $\lambda_i=\lambda_{i+1}$, then
the columns of $F_{i,i+1}$
are linearly  independent.
}
\end{equation}
Let $\widetilde M_{pq}'$ be obtained from $M_{pq}'$ by replacing all $F_{ij}$ with $\widetilde F_{ij}$ of the same sizes such that the columns of $\widetilde F_{i,i+1}$
are linearly  independent
if $\lambda_i=\lambda_{i+1}.$ The equality $S_{pp}M_{pq}'=\widetilde M_{pq}'S_{pp}$ implies that
$S_{pp}=S_1\oplus\dots\oplus
S_{k}$, in which each $S_i$ has the same size as $\lambda_iI$ (see \cite[Lemma 2.1(b)]{ser2}). The algebra $\Lambda'$ consists of those matrices $S\in\Lambda$, in which $S_{pp}$ has the form $S_1\oplus\dots\oplus
S_{k}$. Note that in this subcase we have reduced only the diagonal blocks $\lambda_1I,\dots,\lambda_kI$ and the zero blocks of $M_{pq}'$ under them; the blocks $F_{ij}$ are reduced in next steps of the algorithm.
\end{description}

\begin{theorem}           \label{t5.2}
Let $\Gamma \subset {\mathbb C}^{t\times t}$ be a basic reduced algebra, and let  ${\cal U} \subset \{1,\dots,t\}$ be closed under the equivalence $\approx$ from \eqref{0}. Let $\un n=(n_1,\dots,n_t)$ be a sequence of nonnegative integers such that $n_{i} = n_{j}$ if $i \approx
j$. Then each $\un n\times \un n$ complex matrix $M$ is $(\Gamma^{\un n\times \un n},\mc U)$-equivalent to the matrix $M_{\text{\rm can}}$ constructed by the algorithm. Two $\un n\times \un n$ complex matrices $M$ and $N$ are $(\Gamma^{\un n\times \un n},\mc U)$-equivalent if and only if $M_{\text{\rm can}}= N_{\text{\rm can}}$.
\end{theorem}

\begin{proof}
Write $\Lambda:=\Gamma^{\un n\times \un n}$. By the algorithm, \[M\sim_{(\Lambda,\mc U)} M'\sim_{(\Lambda',\mc U')}M''\sim_{(\Lambda'',\mc U'')}\cdots.\] Hence, $M\sim_{(\Lambda,\mc U)} M_{\text{can}}$.

Let $M\sim_{(\Lambda,\mc U)}N$. If $M_{pq}$ is the first nonstable block of $M$ with respect to the  ordering \eqref{5}, then $N_{pq}$ is the first nonstable block of $N$; the preceding blocks of $M$ and $N$ coincide.  If the equalities \eqref{3} do not imply \eqref{7}, then $M'_{pq}=N'_{pq}=0$. If they
imply \eqref{7}, then $M_{pq}S_{qq}= S_{pp}N_{pq}$, and so $M'_{pq}=N'_{pq}$ in each of Subcases II(a)--II(e). In Subcase II(f),
$M'_{pq}$ and $N'_{pq}$ have the same form \eqref{5.10r} up to blocks $F_{ij}$.
Hence, $M$ and $N$ define the same pair $(\Lambda',\mc U')$, and $M'\sim_{(\Lambda',\mc U')}N'$. We repeat this reasoning until we obtain
$M^{(k)}\sim_{(\Lambda^{(k)},\mc U^{(k)})}N^{(k)}$, where $k$ as in \eqref{ksj}. Since $M^{(k)}$ is $(\Lambda^{(k)},\mc U^{(k)})$-similar only to itself, we have $M^{(k)}=N^{(k)}$, and so
$M_{\text{can}}=N_{\text{can}}$.
\end{proof}

\section*{Acknowledgements}
The work was supported in part by the UAEU UPAR grant G00002160. V.V.~Sergeichuk was also supported by the FAPESP grant 2018/24089-4.


\begin{thebibliography}{99}
\bibitem{azi}
T.Ya.  Azizov,  I.S.  Iokhvidov,
Linear  Operators in Spaces with an  Indefinite  Metric,  John  Wiley  \& Sons, Ltd., Chichester, 1989.

\bibitem{bel} G.R. Belitski\u\i,
    Normal forms in a space of
    matrices, in: V.A. Marchenko (Ed.), Analysis in
    Infinite Dimensional Spaces and
    Operator Theory, Naukova Dumka, Kiev, 1983,
    pp. 3--15 (in Russian).

\bibitem{bel1} G.R. Belitskii,
    Normal forms in matrix spaces,
    Integral Equations Operator
    Theory 38 (2000) 251--283.

\bibitem{bel-ser_comp} G.R. Belitskii,
    V.V. Sergeichuk, Complexity of matrix
    problems, Linear Algebra
    Appl. 361 (2003) 203--222.

\bibitem{ben}
R. Benedetti, P. Cragnolini, Versal families of matrices with respect to unitary    conjugation, Adv. in Math. 54 (1984) 314--335.

\bibitem{bog}
J. Bognar, Indefinite Inner Product Spaces,
Springer-Verlag, New York-Heidelberg, 1974.

\bibitem{br}
T. Bruestle, V.V. Sergeichuk,
Estimate  of  the  number  of  one-parameter  families  of  modules  over  a  tame  algebra, Linear Algebra Appl. 365 (2003) 115--133.

\bibitem{ch}
Y. Chen, L. Nie, Y. Xu, Belitskii's canonical forms of linear dynamical systems, Linear Algebra Appl. 531 (2017) 533--536.

\bibitem{ch1}
Y. Chen,  Y. Xu, H. Li, W. Fu, Belitskii's canonical forms of upper triangular nilpotent matrices under upper triangular similarity, Linear Algebra Appl. 506 (2016) 139--153.

\bibitem{gab_vos}
P. Gabriel, L.A. Nazarova, A.V. Roiter, V.V. Sergeichuk, D. Vossieck, Tame
and wild subspace problems,
Ukrainian Math. J. 45 (1993)
335--372.


\bibitem{gan1}
F.R. Gantmacher,  The Theory of Matrices, Vol. 2, AMS Chelsea Publishing, Providence, RI, 2000.

\bibitem{goh}
 I. Gohberg, P. Lancaster, L. Rodman, Indefinite Linear Algebra and Applications, Birkh\"auser, Boston, 2006.

\bibitem{hor}
R.A. Horn, C.R. Johnson, Matrix Analysis, 2nd ed., Cambridge University Press, Cambridge, 2013.

\bibitem{lit}
D.E. Littlewood, On unitary equivalence, J. London Math. Soc.   28 (1953) 314--322.

\bibitem{meh1}
C. Mehl, A.C.M. Ran, L. Rodman, Hyponormal matrices and semidenite invariant subspaces in indefinite inner products, Electron. J. Linear Algebra 11 (2004) 192--204.

\bibitem{meh}
C.  Mehl,  A.C.M.  Ran,  L.  Rodman,  Semidefinite  invariant  subspaces:  degenerate inner products, Operator Theory:  Advances and Appl., Current trends in operator theory and its applications, 149 (2004) 467--486.

\bibitem{meh2}
C. Mehl, L. Rodman, Symmetric matrices with respect to sesquilinear forms,
Linear Algebra Appl. 349 (2002) 55--75.

\bibitem{ser1} V.V. Sergeichuk, Classification of linear operators in a finite dimensional unitary     space, Funct. Anal. Appl.  18 (3) (1984) 224--230.

\bibitem{ser_izv}
V.V. Sergeichuk,
Classification
problems  for systems
of forms and linear
mappings, Math.
USSR-Izv. 31 (no. 3)
(1988) 481--501. Theorem 2 over quaternions was corrected in arxiv.org/abs/0801.0823.

\bibitem{ser2}
V.V. Sergeichuk, Unitary and Euclidean representations of
a quiver, Linear Algebra Appl. 278 (1998) 37--62.

\bibitem{ser_can}
V.V. Sergeichuk,
Canonical matrices for linear matrix problems, Linear Algebra Appl. 317 (2000) 53--102.

\bibitem{gal}
V.V. Serge\u\i chuk,
D.V. Galinski\u\i,
Classification  of  pairs  of  linear  operators  in  a  four-dimensional  vector  space, Infinite groups and related algebraic structures, Akad. Nauk Ukrainy, Inst.
Mat., Kiev, 1993, 413--430 (in Russian).

\bibitem{sha}
H. Shapiro, A survey of canonical forms and invariants for unitary similarity, Linear Algebra Appl. 147 (1991) 101--167.
\end{thebibliography}
\end{document}